\documentclass[a4paper]{amsart}
\usepackage{amssymb}
\usepackage{amscd}
\usepackage{mathtools}
\usepackage{amsthm}
\usepackage{amsmath}
\usepackage{latexsym}
\usepackage[all]{xy}

\theoremstyle{plain}
\newtheorem{theorem}{Theorem}
\newtheorem{corollary}[theorem]{Corollary}
\newtheorem{lemma}[theorem]{Lemma}
\newtheorem{proposition}[theorem]{Proposition}

\theoremstyle{definition}
\newtheorem{definition}[theorem]{Definition}
\newtheorem{construction}[theorem]{Construction}
\newtheorem{reminder}[theorem]{Reminder}

\newtheorem{example}[theorem]{Example}

\theoremstyle{remark}

\numberwithin{theorem}{section} 

\newenvironment{ii}
{ \begin{enumerate}}
{\end{enumerate}}

\newcommand{\Spec}[1]{\mbox{\rm{Spec}}(#1)}

\newcommand{\Add}{\mbox{\rm{Add\,}}}

\newcommand{\Gen}{\mbox{\rm{Gen\,}}}
\newcommand{\Cogen}{\mbox{\rm{Cogen\,}}}

\newcommand{\Soc}[1]{\mbox{\rm{Soc}} \,#1}

\newcommand{\im}{\mbox{\rm{Im\,}}}

\newcommand{\Hom}[3]{\mbox{\rm{Hom}}_{#1}(#2,#3)}
\newcommand{\Ext}[4]{\mbox{\rm{Ext}}^{#1}_{#2}(#3,#4)}

\newcommand{\Tra}[1]{\mbox{\rm{Tr}}#1}

\newcommand{\rmod}[1]{\mbox{\rm{Mod}--}{#1}}

\newcommand{\ModR}{\text{Mod-}R}
\newcommand{\RMod}{R\text{-Mod}}
\newcommand{\Rmod}{R\text{-mod}}
\newcommand{\ProjR}{\text{Proj-R}}
\newcommand{\projR}{\text{proj-R}}
\newcommand{\Rproj}{\text{R-proj}}
\newcommand{\Mor}[1]{\mbox{\rm{Mor}}(#1)}
\newcommand{\Div}[1]{\mbox{\rm{Div}-}{#1}}

\newcommand{\Ann}[1]{\mbox{\rm{Ann}}(#1)}

\newcommand{\Ker}[1]{\mbox{\rm{Ker}}\,#1}
\newcommand{\Coker}{\mbox{\rm{Coker}}}

\newcommand{\C}{\mathcal{C}}

\newcommand{\Z}{\mathbb{Z}}

\title{Silting modules over commutative rings}
\author{Lidia Angeleri H\"{u}gel}
\email{lidia.angeleri@univr.it}
\author{Michal Hrbek}
\email{hrbmich@gmail.com}
\thanks{The first named author is partially supported by Fondazione Cariparo, Progetto di Eccellenza ASATA. The second named author is partially supported by the Grant Agency of the Czech Republic under the grant no.~14-15479S. Preliminary version of \today} 

\begin{document}
\maketitle
\begin{abstract}
Tilting modules over commutative rings were  recently classified in \cite{H}: they correspond  bijectively to faithful Gabriel topologies of finite type. In this note we  extend this classification by dropping faithfulness. The counterpart of an arbitrary Gabriel topology of finite type is obtained by replacing tilting with the more general notion of a silting module. 
\end{abstract}
\section{Introduction}
Silting modules were introduced in \cite{AMV} as a common generalisation of tilting modules and of the support $\tau$-tilting modules from \cite{AIR}. They  are in bijection with 2-term silting complexes and with certain t-structures and co-t-structures in the derived module category. For certain rings, they are also known to parametrize universal localisations  and wide subcategories of finitely presented modules  \cite[Theorem 4.5]{MS1},\cite[Corollary 5.15]{AMV2}.

In this note, we give a classification of silting modules over commutative rings, establishing a bijective correspondence with Gabriel filters of finite type. This extends the results in \cite{H} from the tilting to the silting case, and it is  a further piece of evidence for  the close relationship between silting modules and localisation theory. 

Our result  is achieved by investigating the dual notion of a cosilting module recently introduced in \cite{BP} as a  generalisation of cotilting modules. 
Indeed, the dual of a silting module $T$ is a cosilting module $T^+$, and there is a duality between the modules in the silting class $\Gen T$ and the  cosilting class $\Cogen T^+$. When $R$ is commutative, $\Cogen T^+$ turns out to be the torsionfree class of a hereditary torsion pair of finite type. We can thus interpret the modules in $\Cogen T^+$ as the $\mathcal{G}$-torsionfree modules with respect to a Gabriel filter of finite type $\mathcal{G}$. The silting class $\Gen T$ is then the class of $\mathcal{G}$-divisible modules. This defines a map assigning a  Gabriel filter $\mathcal{G}$ to every  silting class $\Gen T$. We show that this assignment is a bijection by constructing explicitly, for any   $\mathcal{G}$, a silting module $T$ which generates the $\mathcal{G}$-divisible modules (Construction~\ref{CC00}). We also provide a construction for a cosilting module cogenerating the $\mathcal{G}$-torsionfree modules (Construction~\ref{coscon}).

In general, not all cosilting modules  arise as duals of silting modules. This is a phenomenon that already occurs for cotilting modules \cite{B}, see Example \ref{E00} for a cosilting example. If $R$ is a commutative noetherian ring, however, our classification yields  bijections between silting classes, cosilting classes, Gabriel filters, and subsets of $\Spec{R}$ closed under specialisation (Theorem~\ref{classcosilting}). This generalises the classification of tilting and cotilting modules in \cite[Theorem 2.11]{APST}.

In fact,  silting and cosilting classes are in bijection also over non-commutative noetherian rings. As a consequence, every definable torsion class of right modules  over a left noetherian ring  is generated by a silting module (Corollary~\ref{coftnoe}). Finally,  extending a result from \cite{BHPST}, we show that the only silting torsion pair of finite type over a commutative ring is the trivial one  (Proposition~\ref{P11}).

The note is organized as follows. In Section 2 we investigate a finiteness condition which is shown  to hold for silting classes, recovering  a recent result from \cite{MS}.  Section 3  is devoted to the duality between silting and cosilting classes. In Sections 4 and 5 we turn to commutative rings and prove our classification results. In \ref{E00} we further exhibit an example showing that the inclusion of silting modules in the class of finendo quasitilting modules  proved in \cite[Proposition 3.10]{AMV} is proper.

\section{Definability and finite type}

Let $R$ be a ring, and let $\ModR$ (respectively, $\RMod$) denote the category of all right (respectively, left) $R$-modules. Denote by $\ProjR$ and $\projR$ the full subcategory of $\ModR$ consisting of all projective and all finitely generated projective modules, respectively. Given a subcategory $\mathcal{C}$ of $\ModR$, write $\Mor{\mathcal{C}}$ for the class of all morphisms in $\ModR$ between objects in $\mathcal{C}$, and denote $$\mathcal{C}^\perp=\{M \in \ModR \mid \Ext{1}{R}{C}{M}=0\}.$$

 Given a map $\sigma$  in
$\Mor{\ProjR}$, we are going to investigate the class 
$$\mathcal{D}_\sigma:=\{X\in \rmod R \,\mid\, \Hom{R}{\sigma}{X}\ \text{is surjective}\}.$$
We say that $\mathcal{D}_\sigma$ is of \emph{finite type} if it is determined by a set of  morphisms between finitely generated projective modules, i.~e.~there are $\sigma_i \in \Mor{\projR}, i\in I,$ such that $\mathcal{D}_\sigma = \bigcap_{i \in I} \mathcal{D}_{\sigma_i}$. As a shorthand, we say that $\sigma \in \Mor{\ProjR}$ is of \emph{finite type} if the class $\mathcal{D}_\sigma$ is of finite type.

Recall that a class is said to be \emph{definable}  if it is closed under direct limits, direct products and pure submodules.
We are going to  see that $\mathcal{D}_\sigma$ is definable if and only if it is of finite type.

\begin{lemma}
		\label{L00}
		Let $\sigma \in \Mor{\ProjR}$. Then $$\mathcal{D}_\sigma=(\Coker{\sigma})^\perp \cap \mathcal{D}_{\sigma'},$$ where $\sigma': P_{-1} \rightarrow \im{\sigma}$ is given by restricting the codomain of $\sigma$ to its image.
\end{lemma}
\begin{proof}
		It is clear that $\mathcal{D}_{\sigma} \subseteq \mathcal{D}_{\sigma'}$. Then for any $M \in \mathcal{D}_{\sigma'}$, a standard long exact sequence argument shows that $M \in \mathcal{D}_\sigma$ if and only if $\Ext{1}{R}{\Coker{\sigma}}{M}=0$, finishing the proof.
\end{proof}

\begin{lemma}
	\label{L01}
	Let $\sigma \in \Mor{\ProjR}$ be a map between projective modules. Then $\mathcal{D}_\sigma=\mathcal{D}_{\varphi}$, where $\varphi$ is a map between free modules.
\end{lemma}
\begin{proof}
		Suppose that $\sigma: P_{-1} \rightarrow P_0$. Let $P'$ be a projective module such that $P_{-1} \oplus P$ is free, and let $P''$ be a projective module such that $P_0 \oplus P' \oplus P''$ is free. We then let $\varphi$ be the direct sum of the maps $\sigma: P_{-1} \rightarrow P_0, P' \xrightarrow{=} P'$, and $0 \xrightarrow{0} P''$. It is a routine check that $\mathcal{D}_\sigma = \mathcal{D}_{\varphi}$.
\end{proof}

\begin{theorem}\label{deft}
		\label{T00}
		Let $\sigma \in \Mor{\ProjR}$. Then the following are equivalent:
		\begin{enumerate}
			\item[(i)] 	$\mathcal{D}_\sigma$ is of finite type,		
			\item[(ii)] $\mathcal{D}_\sigma$ is definable.
		\end{enumerate}
\end{theorem}
\begin{proof}
		In the whole proof, let $\sigma: P_{-1} \rightarrow P_0$, and $C=\Coker{\sigma}$.

(i) $\rightarrow$ (ii): As an intersection of definable classes is a definable class, it is enough to show that $\mathcal{D}_\sigma$ is definable if $\sigma \in \Mor{\projR}$. By \cite[Lemma 3.9]{AMV}, $\mathcal{D}_\sigma$ is closed under direct products and epimorphic images, it is thus enough to show that it is closed under direct sums and pure submodules. By Lemma~\ref{L00}, we have that $\mathcal{D}_\sigma=\mathcal{D}_{\sigma'} \cap C^\perp$, where $\sigma': P_{-1} \rightarrow \im{\sigma}$ is $\sigma$ with codomain restricted to its image. As $C$ is finitely presented, the class $C^\perp$ is definable by \cite[Theorem 13.26]{GT}. We finish the proof by showing that $\mathcal{D}_{\sigma'}$ is closed under direct sums and submodules.

								Let $(M_i \mid i \in I)$ be a family of modules from $\mathcal{D}_{\sigma'}$, and $f: P_{-1} \rightarrow \bigoplus_{i \in I}M_i$ a map. As $P_{-1}$ is finitely generated, there is a finite subset $J \subseteq I$ such that $f$ factors through the direct summand $\bigoplus_{i \in J}M_i \simeq \prod_{i \in J}M_i$. Since $\mathcal{D}_{\sigma'}$ is clearly closed under products, $f$ factorizes through $\sigma'$.

								Let $M \in \mathcal{D}_{\sigma'}$ and $\iota: N \subseteq M$ be an inclusion. Applying $\Hom{R}{-}{\iota}$ on the exact sequence $0 \rightarrow K \rightarrow P_{-1} \xrightarrow{\sigma'} \im{\sigma} \rightarrow 0$ yields
								$$
									\begin{CD}
											0 @>>> \Hom{R}{\im{\sigma}}{M} @>\Hom{R}{\sigma'}{M}>> \Hom{R}{P_{-1}}{M} @>\varphi>> \Hom{R}{K}{M}  \\
										& & @AAA @AAA @A\theta AA & & \\
											0 @>>> \Hom{R}{\im{\sigma}}{N} @>\Hom{R}{\sigma'}{N}>> \Hom{R}{P_{-1}}{N} @>\psi>> \Hom{R}{K}{N}. \\
									\end{CD}
								$$
						By the assumption, the map $\varphi=0$, and thus $\theta\psi=0$. By left-exactness, all the vertical maps are injective, and therefore $\psi=0$, showing that $\Hom{R}{\sigma'}{N}$ is surjective. Therefore, $N \in \mathcal{D}_{\sigma'}$.

		(ii) $\rightarrow$ (i): Using Lemma~\ref{L01}, we can without loss of generality assume that $P_{-1}$ and $P_0$ are free modules. Fix a free basis $X$ of $P_{-1}$, and write the set $X$ as a direct union $X=\bigcup_{i \in I}X_i$ of its finite subsets, inducing a presentation of $P_{-1}$ as a direct union of direct summands $F_i = R^{(X_i)}$. Denote $G_i=\sigma(F_i)$. Fix a free basis $Y$ of $P_0$. As $G_i$ is finitely generated for each $i \in I$, there is a finite subset $Y'_i$ of $Y$ spanning $G_i$. Moreover, there are finite subsets $Y_i$ of $Y$ such that $Y'_i \subseteq Y_i$ for each $i \in I$, and $(Y_i \mid i \in I)$ forms a directed system. Indeed, we can find such sets by setting $Y_i=Y'_i \cup \bigcup_{j \leq i} Y'_j$. As this is clearly a finite union of finite sets, we have that $(Y_i \mid i \in I)$ is a directed system of finite subsets of $Y$. We let $F'_i=R^{(Y_i)}$, a free direct summand of $P_0$ for each $i \in I$. 

		The directed union $\bigcup_{i \in I}F'_i=R^{(\bigcup_{i \in I}Y_i)}$ is a direct summand of $P_0$, and the projection of the image of $\sigma$ onto the complement $R^{(Y \setminus \bigcup_{i \in I}Y_i)}$ is necessarily zero. Therefore, we can without loss of generality assume that $P_0=\bigcup_{i \in I}F'_i$. Let $\sigma_i: F_i \rightarrow F'_i$ be the restriction of $\sigma$ onto $F_i$, with codomain restricted to $F'_i$. We claim that $\mathcal{D}_\sigma \subseteq \mathcal{D}_{\sigma_i}$ for each $i \in I$. To prove this, let $M \in \mathcal{D}_\sigma$ and fix a map $f_i: F_i \rightarrow M$. As $F_i$ is a direct summand of $P_{-1}$, we can extend $f_i$ to a map $f: P_{-1} \rightarrow M$. As $M \in \mathcal{D}_\sigma$, there is a map $g: P_{0} \rightarrow M$ such that $f=g\sigma$. Let $g_i$ be the restriction of $g$ to $F'_i$. Then $f_i=g_i\sigma_i$, proving that $M \in \mathcal{D}_{\sigma_i}$. Denoting $\mathcal{D}=\bigcap_{i \in I}\mathcal{D}_{\sigma_i}$, we have $\mathcal{D}_\sigma \subseteq \mathcal{D}$.

		Finally, we show that $\mathcal{D}_\sigma$ is of finite type by proving $\mathcal{D} \subseteq \mathcal{D}_\sigma$. The class $\mathcal{D}_\sigma$ is definable by the assumption, and the definability of the class $\mathcal{D}$ is proved by implication $(i) \rightarrow (ii)$ of this Theorem. By \cite[Lemma 6.9]{GT}, it is enough to show that $M \in \mathcal{D}$ implies $M \in \mathcal{D}_\sigma$ for $M$ pure-injective. 
		
		Let $M \in \mathcal{D}$ be pure-injective. Denote  $C_i=\Coker{\sigma_i}$ for all $i \in I$. By Lemma~\ref{L00}, we have that $\Ext{1}{R}{C_i}{M}=0$ and $M \in \mathcal{D}_{\sigma_i'}$, where $\sigma'_i$ is given by restricting the codomain of $\sigma_i$ to $G_i$. Since $M$ is pure-injective, we have by \cite[Lemma 6.28]{GT} the following isomorphism:
		$$\Ext{1}{R}{C}{M} \simeq \Ext{1}{R}{\varinjlim_{i \in I} C_i}{M} \simeq \varprojlim_{i \in I} \Ext{1}{R}{C_i}{M}.$$
		This shows that $M \in C^\perp$. Applying $\Hom{R}{-}{M}$ to the exact sequence $F_i \xrightarrow{\sigma'_i} G_i \rightarrow 0$ we obtain that $\Hom{R}{\sigma'_i}{M}$ is an isomorphism for all $i \in I$. As inverse limit of a directed system of isomorphisms is an isomorphism, we obtain that 
		$$\varprojlim_{i \in I} \Hom{R}{\sigma'_i}{M} \simeq \Hom{R}{\varinjlim_{i \in I} \sigma'_i}{M} \simeq \Hom{R}{\sigma'}{M}$$
		is an isomorphism, where again $\sigma': P_{-1} \rightarrow G$ is given by restricting the codomain of $\sigma$ to its image $G=\bigcup_{i \in I}G_i$. In other words, $M \in \mathcal{D}_{\sigma'}$. As $M \in C^\perp \cap \mathcal{D}_{\sigma'}$, Lemma~\ref{L00} yields $M \in \mathcal{D}_{\sigma}$ as desired.
						\end{proof}

\section{Silting  and cosilting modules}
According to \cite{AMV}, an $R$-module $T$ is said to be {\it silting} if it admits a projective presentation  $P_{-1}\stackrel{\sigma}{\longrightarrow} P_0\longrightarrow T\to 0$ such that the class $\Gen T$ of $T$-generated modules coincides with the class $\mathcal{D}_\sigma$.
 The class $\Gen T$ is then called a {\it silting class}.

It is shown in \cite[3.5 and 3.10]{AMV} that silting classes are  definable torsion classes. From Theorem \ref{T00} we obtain that every silting class is of finite type. This reproves one implication in  a recent result due to  Marks and \v{S}\v{t}ov\'{i}\v{c}ek.
\begin{theorem}\label{ft} \cite{MS} A map $\sigma$ in $\Mor{\ProjR}$ is of finite type if and only if  the  class $\mathcal{D}_\sigma$  is a  silting class.
\end{theorem}

\smallskip

Let us now turn to the dual notion. 
Following \cite{BP}, an $R$-module $C$ is said to be {\it cosilting} if it admits an injective copresentation  $0\to C\longrightarrow E_{0}\stackrel{\sigma}{\longrightarrow} E_1$ such that the class $\Cogen C$ of $C$-cogenerated modules coincides with the class 
$$\mathcal{C}_\sigma:=\{X\in \rmod R \,\mid\, \Hom{R}{X}{\sigma}\ \text{is surjective}\}.$$ The class $\Cogen C$ is then called a {\it cosilting class}.

It is shown  in \cite{BP} that every cosilting module is pure-injective and that cosilting classes are definable torsionfree classes. 
In fact,  there is a duality between the silting classes in $\ModR$ and certain cosilting classes in $\RMod$ (see also \cite[3.7 and 3.9]{BP}). These cosilting classes will be characterized by  the property below.

\begin{definition}
		For any  map $\sigma \in \Mor{\ProjR}$, denote 
		$$\mathcal{T}_\sigma = \{X \in \RMod \mid \sigma \otimes_R X \text{ is injective}\}.$$
		Given a map $\lambda$ between injective left $R$-modules, we say that the class $\mathcal{C}_\lambda$ (or, the map $\lambda$) is of \emph{cofinite type}, if there is a set of maps $\sigma_i \in \Mor{\projR}, i \in I,$ such that $\mathcal{C}_\lambda=\bigcap_{i \in I} \mathcal{T}_{\sigma_i}$.
\end{definition}

Let us  investigate the duality. Assume that $R$ is a $k$-algebra over some commutative ring $k$. Given an $R$-module $M$, we denote by $M^+$ its dual with respect to an injective cogenerator of $\rmod k$, for example we can take $k=\Z$ and $M^+$ the character dual of $M$.
To every definable category $\mathcal{C}$ of
right (left) $R$-modules we can now associate a \emph{dual definable
category} of left (right) $R$-modules $\mathcal{C} ^\vee$
which is determined by the property that a module $M$ belongs to $\mathcal{C}$ if
and only if its dual module $M^+\in \mathcal{C} ^\vee$. This assignment defines a  bijection between definable subcategories of $\ModR$ and
 $\RMod$,  which restricts to a bijection between definable torsion classes  and definable torsionfree classes
 and  maps tilting classes to cotilting classes of cofinite type, see \cite[Propositions 5.4 and  5.7 and Theorem 7.1]{B}. We are now going to prove the analogous result for silting and cosilting classes.

\begin{lemma}\label{dual}		
	\begin{enumerate}
\item Let $\sigma \in \Mor{\ProjR}$. Then $\mathcal{T}_\sigma= \mathcal{C}_{\sigma^+}$, and a left $R$-module $X$ belongs to  $\mathcal{C}_{\sigma^+}$ if and only if $X^+\in \mathcal{D}_{\sigma}$.
\item If $\sigma \in \Mor{\ProjR}$ has finite type, then  
	 $\mathcal{D}_\sigma$ and $\mathcal{C}_{\sigma^+}$ are dual definable categories, and a right $R$-module $Y$  belongs to $ \mathcal{D}_\sigma$ if and only if $Y^+ \in \mathcal{C}_{\sigma^+}$.
\item A map $\lambda$  between injective left $R$-modules has cofinite type if and only if  there is a map $\sigma \in \Mor{\ProjR}$ of finite type such that $\mathcal{C}_\lambda=\mathcal{C}_{\sigma^+}$.	 
\end{enumerate}
		\end{lemma}
\begin{proof}
(1),(2) By Hom-$\otimes$-adjunction, for any left $R$-module $X$ there is a commutative diagram linking the maps $\Hom{R}{X}{\sigma^+}$, $(\sigma\otimes_R X)^+$ and $\Hom{R}{\sigma}{X^+}$. This shows  that $X\in \mathcal{C}_{\sigma^+}$ if and only if $X^+\in \mathcal{D}_{\sigma}$, which in turn means that $(\sigma\otimes_R X)^+$ is surjective, or equivalently, $\sigma\otimes_R X$ is injective.

Furthermore, if $\sigma$ is of finite type, the definable class $\mathcal{D}_{\sigma}$ contains a right $R$-module $Y$ if and only if it contains its double dual $Y^{++}$, see e.g.~\cite[3.4.21]{P}. This implies that $Y\in \mathcal{D}_{\sigma}$ if and only if $Y^+\in \mathcal{C}_{\sigma^+}$.

(3)  Let $\sigma_i \in \Mor{\projR}, i \in I,$ be a set of maps such that $\mathcal{C}_\lambda=\bigcap_{i \in I}\mathcal{T}_{\sigma_i}$, and let $\sigma=\bigoplus_{i\in I} \sigma_i$. Then
		$\mathcal{C}_\lambda=\bigcap_{i \in I}\mathcal{T}_{\sigma_i}=\mathcal{T}_\sigma=\mathcal{C}_{\sigma^+}$ by (1), and
		 $\mathcal{D}_\sigma=\bigcap_{i \in I}\mathcal{D}_{\sigma_i}$, so the map $\sigma$ is of finite type.
		Conversely, if $\mathcal{C}_{\lambda}=\mathcal{C}_{\sigma^+}$ for a map $\sigma$  of finite type, there are maps $\sigma_i \in \Mor{\projR}, i \in I,$ such that $\mathcal{D}_\sigma=\bigcap_{i \in I}\mathcal{D}_{\sigma_i}$, and  
				$\mathcal{C}_{\lambda}=\mathcal{C}_{\sigma^+}=\bigcap_{i \in I}\mathcal{C}_{\sigma_i^+}=\bigcap_{i \in I}\mathcal{T}_{\sigma_i}.$
\end{proof}

\begin{proposition} \label{dualdefinable}
		Let $\sigma \in \Mor{\ProjR}$, and let  $T=\Coker{\sigma}$ be a silting module with respect to $\sigma$. Then $T^+$ is a cosilting left $R$-module  with respect to the injective copresentation   $\sigma^+$. Moreover, $\Gen T$ and $\Cogen T^+$ are dual definable classes, and $\Cogen T^+$ is a cosilting class of cofinite type.
		\end{proposition}
\begin{proof}
We have to verify  $\Cogen T^+=\mathcal{C}_{\sigma^+}$.  The class $\mathcal{C}_{\sigma^+}$ is closed under submodules by \cite[3.5]{BP}, so for the inclusion $\subset$  it is enough to show that $\mathcal{C}_{\sigma^+}$ contains the direct product $(T^+)^\alpha$ for any cardinal $\alpha$. Notice that the definable class $\mathcal{D}_{\sigma}$ contains  $T^{(\alpha)}$. The claim then follows from  Lemma \ref{dual} as $T^{(\alpha)}\,^{+}\cong (T^+)^\alpha$.  
For the inclusion $\supset$, take $X\in\mathcal{C}_{\sigma^+}$. Then $X^+\
\in \mathcal{D}_{\sigma}=\Gen T$, so there is an epimorphism $T^{(\alpha)}\to X^+$ for some cardinal $\alpha$. This yields a monomorphism $X\hookrightarrow X^{++}\to (T^+)^\alpha$, showing that $X\in\Cogen T^+$.
\end{proof}

From  Theorem \ref{ft} and Lemma \ref{dual} we obtain  
\begin{corollary}\label{duality} The assignment $\Gen T\mapsto \Cogen T^+$ defines a 1-1-correspondence between silting classes  in $\ModR$ and cosilting classes of cofinite type in $\RMod$. \end{corollary}

We now give a criterion for a torsionfree definable class to be  of cofinite type.
\begin{lemma}\label{coft}
Let $\mathcal{U}$ be a set of finitely presented left $R$-modules, and let $(\mathcal{T},\mathcal{F})$ be the torsion pair in $\RMod$ generated by $\mathcal{U}$, that is, $\mathcal{F}=\{M\in \RMod\,\mid\, \Hom{R}{U}{M}=0 \text{ for all } U\in \mathcal{U}\}.$ Then $\mathcal{F}$ is a cosilting class of cofinite type.
\end{lemma}
\begin{proof}
For every $U\in\mathcal{U}$ we choose a projective presentation $\alpha_U\in\Mor\Rproj$, and we denote $\sigma_U=\alpha_U\,^\ast$ and $\sigma=\bigoplus_{U\in \mathcal{U}} \sigma_U$. Then, using that for any  $P\in\Rproj$ and any $X\in\RMod$  there is a   natural isomorphism $P^\ast\otimes_R X\cong \Hom{R}{P}{X}$, we see that $\mathcal{F}=\bigcap_{U \in \mathcal{U}} \mathcal{T}_{\sigma_U}=\mathcal{C}_{\sigma^+}$ is a cosilting class of cofinite type.
\end{proof}

\begin{corollary}\label{coftnoe}
If $R$ is a left noetherian ring, the definable torsionfree classes in $\RMod$
coincide with the  cosilting classes of cofinite type, and the assignment $\Gen T\mapsto \Cogen T^+$ defines a 1-1-correspondence between silting classes  in $\ModR$ and cosilting classes in $\RMod$.
Moreover,  the definable torsion classes in $\ModR$
coincide with the  silting classes.  
\end{corollary}
\begin{proof}
Let  $(\mathcal{T},\mathcal{F})$ be a torsion pair in $\RMod$ with $\mathcal{F}$ being definable. By \cite[Lemma 4.5.2]{GT}, there is a torsion pair $(\mathcal{U},\mathcal{V})$ in $\Rmod$ such that $\mathcal{T}$ and $\mathcal{F}$ consist of the direct limits of modules in $\mathcal{U}$ and $\mathcal{V}$, respectively, and  $\mathcal{F}=\{M\in \RMod\,\mid\, \Hom{R}{U}{M}=0 \text{ for all } U\in \mathcal{U}\}$. Then $\mathcal{F}$ is a cosilting class of cofinite type by Lemma   \ref{coft}. In particular, every cosilting class is of cofinite type, and Corollary \ref{duality} yields the second statement. 

For the last statement, recall from \cite[Proposition 5.7]{B} that the bijection in Corollary \ref{duality} extends to a bijection between definable torsion classes and definable torsionfree classes. By the discussion above, if  $\mathcal{T}$ is a definable torsion class, its dual definable class $\mathcal{T}^+$ coincides with the dual definable class of a silting class. Now use that  the assignment is injective.
 \end{proof}

In general, a definable  torsion class need not be silting, cf.~Example \ref{E00}. As for the dual result, it was recently shown in \cite{WZ} that the definable torsionfree classes over an arbitrary ring are precisely the  cosilting classes. But in general these classes are not of cofinite type, see again Example \ref{E00}.

\section{Silting  classes over commutative rings}
In this section, we classify silting classes over commutative rings, proving that they coincide precisely with the classes of divisibility by sets of finitely generated ideals. 

The key to our classification are the following results relating cosilting modules of cofinite type with hereditary torsion pairs.
Recall that a torsion pair $(\mathcal{T},\mathcal{F})$ is {\it hereditary} if the torsion class $\mathcal{T}$ is closed under submodules, or equivalently, the torsionfree class $\mathcal{F}$ is closed under injective envelopes. Moreover, $(\mathcal{T},\mathcal{F})$ has {\em finite type} if $\mathcal{F}$ is closed under direct limits.

\smallskip

First of all, combining Lemma~\ref{coft} with \cite[Lemma 2.4]{H}, we obtain 
\begin{corollary}\label{hercos}
Let $R$ be a ring. If  $(\mathcal{T},\mathcal{F})$ is a hereditary torsion pair of finite type in $\RMod$, then $\mathcal{F}$ is a cosilting class  of  cofinite type. 
\end{corollary}
For a commutative ring, also the converse holds true.
\begin{lemma}\label{L35}
		Let $R$ be a commutative ring. Let $\lambda$ be a map between injective  $R$-modules. If $\mathcal{C}_\lambda$ is of cofinite type, then it is a torsionfree class in a  hereditary torsion pair of finite type.
		
		In particular, if $R$ is a commutative noetherian ring, a torsion pair has finite type if and only if it is hereditary.
\end{lemma}
\begin{proof}
		By assumption $\mathcal{C}_\lambda=\bigcap_{i \in I}\mathcal{T}_{\sigma_i}$ for a set of maps $\sigma_i \in \Mor{\projR}, i \in I$. It is then enough to prove the claim for each $\mathcal{T}_{\sigma_i}$, or in other words, we can assume w.l.o.g. that $\mathcal{C}_\lambda=\mathcal{T}_{\sigma}$ for some $\sigma\in \Mor{\projR}$. By Lemma~\ref{dual}, $\mathcal{T}_{\sigma}=\mathcal{C}_{\sigma^+}$ is a  definable category, which is closed under submodules and extensions by \cite[Lemma 2.3]{BP}, so it is a  torsion-free class closed under direct limits. It remains to show that is is also closed under injective envelopes.

Let $M \in \mathcal{T}_{\sigma}$, and consider the exact sequence induced by an injective envelope $0 \rightarrow M \xrightarrow{\iota} E(M) \rightarrow C \rightarrow 0$. Tensoring this sequence with $\sigma$ yields a commutative diagram
		$$
		\begin{CD}
				& & 0  & & 0  \\
				   & &	@VVV    @VVV \\
				0 @>>> P_{-1} \otimes_R M @>\sigma \otimes_R M>> P_0 \otimes_R M \\
				& & @VVP_{-1} \otimes_R \iota V    @VVV \\
				& & P_{-1} \otimes_R E(M) @>\sigma \otimes_R E(M)>> P_0 \otimes_R E(M). \\
		\end{CD}
		$$
		The exactness of the columns follows from the projectivity of $P_{-1},P_0$, while the exactness of the first row follows by definition of $ \mathcal{T}_{\sigma}$. Since $R$ is commutative, this is a commutative diagram in $\ModR$ (this is where we need the commutativity of $R$). First, we claim that the inclusion $P_{-1} \otimes_R \iota$ is an injective envelope of $P_{-1} \otimes_R M$. Indeed, let $P$ be a finitely generated projective such that $P_{-1} \oplus P \simeq R^n$ for some $n$. Then $(P_{-1} \oplus P) \otimes_R \iota = R^n \otimes_R \iota$ is essential by \cite[Proposition 6.17(2)]{AF}, and since $E(M)^n \simeq R^n \otimes_R E(M)$ is injective, it is an injective envelope of $M^n=R^n \otimes_R M$. As $R^n \otimes_R \iota = (P_{-1} \otimes_R \iota) \oplus (P \otimes_R \iota)$, we conclude that $P_{-1} \otimes_R \iota$ is an injective envelope of $P_{-1} \otimes_R M$.
		
		If $P_{-1} \otimes_R M$ is zero, then its injective envelope $P_{-1} \otimes_R E(M)$ is also zero, and thus $\sigma \otimes_R E(M)$ is injective. Towards a contradiction, suppose that $P_{-1} \otimes_R M$ is non-zero, and the kernel of $\sigma \otimes_R E(M)$ is non-zero.  By the essentiality of $P_{-1} \otimes_R \iota$, there is a non-zero element $x \in P_{-1} \otimes_R M$ such that $(\sigma \otimes_R E(M))(P_{-1} \otimes_R \iota)(x)=0$, which by a simple diagram chasing yields $(\sigma \otimes_R M)(x)=0$, a contradiction to $\sigma \otimes_R M$ being injective. Therefore, the kernel of $\sigma \otimes_R E(M)$ is zero. In both cases, we showed that $E(M) \in \mathcal{T}_{\sigma}$.
		
		The last statement follows from Corollary~\ref{coftnoe}.
\end{proof}

It is well known that hereditary torsion pairs  correspond bijectively to Gabriel filters. This will  allow to establish a correspondence between silting classes and Gabriel filters. We first review the relevant notions.
\begin{reminder}\label{reminder}	A filter  $\mathcal{G}$ of right ideals of $R$ is a \emph{(right) Gabriel filter}, if the following conditions hold true:
	\begin{enumerate}
			\item[(i)] if $I \in \mathcal{G}$, then for any $x \in R$ the ideal $(I : x) = \{r \in R \mid xr \in I\}$ belongs to $\mathcal{G}$,
			\item[(ii)] if $J$ is a right ideal such that there is $I \in \mathcal{G}$ with $(J : x) \in \mathcal{G}$ for all $x \in I$, then $J \in \mathcal{G}$.
	\end{enumerate}
	Further, $\mathcal{G}$ is of \emph{finite type} if it has a filter basis consisting of finitely generated ideals.
We remark that a filter of ideals of a commutative ring with a filter basis of finitely generated ideals is a Gabriel filter (of finite type) if and only if it is closed under ideal multiplication,
	see e.g. \cite[Lemma 2.3]{H}.

Every Gabriel filter $\mathcal{G}$ induces a hereditary torsion pair $(\mathcal{T}_\mathcal{G},\mathcal{F}_\mathcal{G})$ where $$\mathcal{F}_\mathcal{G}=\bigcap_{I \in \mathcal{G}} \Ker \Hom{R}{R/I}{-}$$ is the class of $\mathcal{G}$-{\em torsionfree modules}. The assignment $\mathcal{G}\mapsto (\mathcal{T}_\mathcal{G},\mathcal{F}_\mathcal{G})$ defines a bijection between Gabriel filters (of finite type) and hereditary torsion pairs (of finite type), see \cite[Chapter VI, Theorem 5.1, and Chapter XIII, Proposition 1.2]{St}.

 Given a Gabriel filter $\mathcal{G}$, we say that a module $M \in \ModR$ is $\mathcal{G}$-{\em divisible} if $ MI=M$ for all $I \in \mathcal{G}$. If $\Div{\mathcal{G}}$ denotes the class of $\mathcal{G}$-divisible modules, then     $$\Div{\mathcal{G}}=\bigcap_{I \in \mathcal{G}} \Ker ({-}\otimes_R{R/I}).$$
 \end{reminder}

 By Hom-$\otimes$ adjunction, a module $M$ is $\mathcal{G}$-divisible if and only if its dual  $M^+$ is $\mathcal{G}$-torsionfree
		  (cf.~\cite[Chapter VI, Proposition 9.2]{St}). So, if the  Gabriel filter  $\mathcal{G}$ is of finite type,   $\Div{\mathcal{G}}$ and $\mathcal{F}_{\mathcal{G}}$ are dual definable classes, and
		 it follows from Corollary \ref{hercos} that $\Div{\mathcal{G}}$ is a silting class.
	
	\smallskip	 
		 
Again,  in the commutative case, we also have  the converse.
\begin{proposition}
		\label{P01}
		Let $R$ be a commutative ring, and let $\sigma \in \Mor{\ProjR}$ be of finite type. Then there is a Gabriel filter of finite type $\mathcal{G}$ such that $\mathcal{D}_\sigma = \Div{\mathcal{G}}$.
\end{proposition}
\begin{proof}
		By assumption $\mathcal{D}_\sigma=\bigcap_{i \in I}\mathcal{D}_{\sigma_i}$ for a set of maps $\sigma_i \in \Mor{\projR}, i\in I$. If  each $\mathcal{D}_{\sigma_i}=\Div{\mathcal{G}_i}$ for some Gabriel filter of finite type $\mathcal{G}_i$, then $\mathcal{D}_\sigma=\bigcap_{i \in I} \Div{\mathcal{G}_i}=\Div{\mathcal{G}}$, where $\mathcal{G}=\{J \subseteq R \mid I_1I_2\cdots I_n \subseteq J \text{ whenever }  I_1,I_2,\ldots,I_n \in \bigcup_{i \in I}\mathcal{G}_i \}$ is the smallest Gabriel filter of finite type containing $\mathcal{G}_i$ for all $i \in I$. So  we can again assume w.l.o.g. that  $\sigma\in \Mor{\projR}$.
		
		 In Lemma~\ref{L35} we showed that $\mathcal{C}_{\sigma^+}$ is a hereditary torsionfree class of finite type. So there is a Gabriel filter $\mathcal{G}$ of finite type such that $\mathcal{C}_{\sigma^+}=\mathcal{F}_\mathcal{G}$,
		which 
		amounts to
		$\mathcal{D}_\sigma=\Div{\mathcal{G}}$.
This completes the proof.\end{proof}

Combining the results above, we obtain the desired classification of silting classes over commutative rings. Here we give a direct proof by providing an explicit  construction for a silting module corresponding to a Gabriel filter of finite type.
It generalises the construction of a Fuchs-Salce tilting module in \cite{H}.
 
\begin{construction}
		\label{CC00}
		Let $R$ be a commutative ring and $\mathcal{G}$ a Gabriel filter of finite type. 
		Let $\mathcal{I}$ be the collection of all finitely generated ideals from $\mathcal{G}$. For each $I \in \mathcal{G}$, we fix a finite set of generators $\{x_0^I,x_1^I,\ldots,x_{n_I-1}^I\}$. 
The projective presentation 
 $$
	   	\begin{CD}
				& & R^{n_I} @>{}>>R @>>> R/I @>>> 0, \\
		\end{CD}
		$$
induces a projective presentation 
$$ 
	   	\begin{CD}
				& & R @>{\sigma_I}>>R^{n_I} @>>> \Tra (R/I) @>>> 0, \\
		\end{CD}
		$$
where 
 $\sigma_I: R \rightarrow R^{n_I}$ is given by  $\sigma_I(1)=(x_0^I,x_1^I,\ldots,x^I_{n_I-1})$ and $\Tra$ denotes   the Auslander-Bridger transpose of $R/I$
(which is uniquely determined only up to  stable equivalence).
It is easy to check that $\mathcal{D}_{\sigma_I}=\{M \in \ModR \mid M=IM\}$, and thus $\Div{\mathcal{G}}=\bigcap_{I \in \mathcal{I}} \mathcal{D}_{\sigma_I}.$

				Let now $\Lambda$ be the set of all finite sequences of pairs $(I,k)$, with $I \in \mathcal{I}$ and $0 \leq k < n_I$. The set includes the empty sequence denoted by $w$, and it is equipped with the operation of concatenation of sequences, for which we use the symbol $\sqcup$. Let $F$ be the free module on basis $\Lambda$, $F'$ the free module on basis $\Lambda \setminus \{w\}$, and $K$ the free module on basis $\Lambda \times \mathcal{I}$. 		
		
				We define a map $\varphi_\mathcal{G}: K \rightarrow F$ by  its values on the designated basis elements: for any $(\lambda, I) \in \Lambda \times \mathcal{I}$ we set
				$$\varphi_\mathcal{G}((\lambda, I))=\lambda - \sum_{k<n_I}x_k^I (\lambda \sqcup (I, k)).$$
		We also define a map $\varphi'_\mathcal{G}: K \rightarrow F'$ by the commutative diagram
		  $$
	   	\begin{CD}
				& & K @>{\varphi_\mathcal{G}}>> F @>>> C_\mathcal{G} @>>> 0 \\
				& & @| p@VVV {}@VVV \\
			    & & K @>{\varphi'_\mathcal{G}}>> F' @>>> C'_\mathcal{G} @>>>0, \\
		\end{CD}
		$$
		where $p$ denotes the canonical projection $F \rightarrow F'$ killing the coordinate $w$.

		 Let now $P_{-1}=K \oplus K$ and $P_0=F \oplus F'$ and consider 
		 		  $$
	   	\begin{CD}
				& & P_{-1} @>{\sigma_\mathcal{G}}>>P_0 @>>> T_\mathcal{G} @>>> 0, \\
		\end{CD}
		$$
where $\sigma_\mathcal{G}$ is the direct sum of the maps $\varphi_\mathcal{G}$ and $\varphi'_\mathcal{G}$, and $T_\mathcal{G}=C_\mathcal{G} \oplus C'_\mathcal{G}$.
\end{construction}

\begin{proposition}
		\label{P02}
		The module $T_\mathcal{G}$ is a silting module with respect to the map $\sigma_\mathcal{G}$, and $\Gen(T_\mathcal{G})=\Div{\mathcal{G}}$.
\end{proposition}
\begin{proof}
		We divide the proof into several steps. Let us first fix some notation. Let $A=\sum_{I \in \mathcal{I}} \Ann{I}$. Further, for every ideal $I\in\mathcal{I}$,  we define 
		$S_I=\Tra(R/I) \otimes_R R/A$, and we set $\mathcal S=\{S_I\,\mid\, I\in\mathcal{I}\}.$

		{\bf Step I:} Every $I\in\mathcal{I}$ gives rise to a faithful ideal $(I+A)/A$ in the ring $R/A$. In other words, every $r\in R$ satisfying  $rI \subseteq A$ must belong to $A$. 
		
		To see this, use that $I$ is finitely generated to find ideals $I_1,I_2,\ldots,I_l \in \mathcal{I}$ such that $rI \subseteq \sum_{j=1}^l \Ann{I_l}$. Then $r(II_1I_2 \cdots I_l)=0$, and  $r \in \Ann{II_1I_2 \cdots I_l} \subseteq A$.

		{\bf Step II:} An $R/A$-module $M$ satisfies $\Ext{1}{R/A}{S_I}{M}=0$ if and only if $M=IM$.

		Indeed,  the map $\sigma_I: R \rightarrow R^{n_I},\, r\mapsto (r\,x_0^I,r\,x_1^I,\ldots,r\,x_{n_I-1}^I)$
		induces  a commutative diagram with exact rows
		 $$
	   	\begin{CD}
				& & R @>{\sigma_I}>> R^{{n_I}} @>>> \Tra (R/I) @>>> 0 \\
				& & {}@VVV {}@VVV {}@VVV \\
				0 @>>> R/A @>{\overline{\sigma_I}} >> (R/A)^{{n_I}} @>>> S_I @>>>0 \\
		\end{CD}
	   $$
	   because the kernel of $\overline{\sigma_I}$, consisting of the elements $\overline{r}\in R/A$ with $rI\subset A$, is trivial by Step I.
	   It is now an easy observation  that  $\Ext{1}{R/A}{S_I}{M}=0$ if and only if $M=(I+A)/A\cdot M=IM$.
 
		{\bf Step III:} Filtration of $C_\mathcal{G}$ and $C'_\mathcal{G}$.

		For each $n<\omega$ denote by $\Lambda_{n}$ the set of all sequences from $\Lambda$ of length at most  $n$. Let $F_n$ be the span of $\Lambda_{n}$ in $F$, and let $G_n$ be the $\varphi_\mathcal{G}$-image of the span of $\Lambda_{n-1} \times \mathcal{I}$ in $K$. For $n=0$ we have $F_0=R\omega$, and we set $G_0=\emptyset$.  Let $C_n$ be the span of the image of $\Lambda_{n}$ in $C$, that is, $C_n=F_n/(F_n \cap G)$, where $G=\im{\varphi_\mathcal{G}}$.

		We claim that $F_n \cap G=AF_n + G_n$. For any $\lambda \in \Lambda_{n}$ and any $I \in \mathcal{I}$, the element $\varphi_\mathcal{G}((\lambda,I))=\lambda - \sum_{k<n_I}x_k^I (\lambda \sqcup (I,k))$ lies in $G$. Therefore, by multiplying by any $r \in \Ann{I}$, we obtain $r\lambda \in G$. As clearly $G_n \subseteq F_n \cap G$, we have $AF_n + G_n \subseteq F_n \cap G$. 
		
		For the reverse inclusion, let $x \in F_n \cap G$. As $x \in G$, it is of the form 
		$$x=\sum_{j=1}^m r_j\varphi_\mathcal{G}((\lambda_j,I_j))=\sum_{j=1}^mr_j(\lambda_j - \sum_{k<n_{I_j}} x_k^{I_j} (\lambda_j \sqcup (I_j,k)))$$ 
		for some $r_j \in R$, and $(\lambda_j, I_j)$ pairwise distinct elements of $\Lambda \times \mathcal{I}$. We claim that if the length of some $\lambda_j$ exceeds $n-1$, then $r_j \in A$. We prove this claim by backward induction on the length of $\lambda_j$. If $j$ is such that the length of $\lambda_j$ is maximal and exceeding $n$, it is clear from $x \in F_n$ that $r_j \in\Ann{\{x_0^{I_j},x_1^{I_j},\ldots,x_{n_{I_j}-1}^{I_j}\}}=\Ann{I_j} \subseteq A$. Suppose now that $\lambda_j$ is of length $k>n-1$, and that all coefficients $r_{i}$ such that $\lambda_i$ has length $>k$ are in $A$. Then, since $x \in F_n$,  the induction premise yields   $r_jx_k^{I_j} \subseteq A$ for each $k=0,1,\ldots,n_{I^j} -1$. In other words, $r_jI \subseteq A$. By Step I, this implies that $r_j \in A$ as claimed. We proved that the coefficient $r_j$ is in $A$ for any $j$ such that the length of $\lambda_j$ exceeds $n-1$, and thus $x \in AF + G_n$. But since $x \in F_n$, and $AF \cap F_n = AF_n$, we get $x \in AF_n + G_n$ as desired.
		
		It follows that $C_n=F_n/(AF_n + G_n)$. Then $C_0 \simeq R/A$, and $C_{n+1}/C_n \simeq F_{n+1}/(F_n + AF_{n+1} + G_{n+1})$ for any $n \in \omega$. Therefore, $C_{n+1}/C_n \simeq F_{n+1}/(F_n + G_{n+1}) \otimes_R R/A$. The elements $\varphi_\mathcal{G}( (\lambda, I))$, where $\lambda$ is of length $n$, and $I \in \mathcal{I}$ generate $G_{n+1}$ modulo $F_n \cap G_{n+1}$. We obtain that $C_{n+1}/C_n$ is isomorphic to:
		$$\bigoplus_{\lambda \in \Lambda_n \setminus \Lambda_{n-1}} \bigoplus_{I \in \mathcal{I}} (R^{(\lambda \sqcup (I,k) \mid k < n_I)}/(\sum_{k<n_I}x_k^I(\lambda \sqcup (I,k)))R) \otimes_R R/A\simeq \bigoplus_{\lambda \in \Lambda_n \setminus \Lambda_{n-1}} \bigoplus_{I \in \mathcal{I}} S_I.$$
		In particular, $C_\mathcal{G}$ is $\{R/A\} \cup \mathcal{S}$-filtered, and the quotient $C_\mathcal{G}/C_0$, which is clearly isomorphic to $C'_\mathcal{G}$, is $ \mathcal{S}$-filtered.
	
	   {\bf Step IV:} We claim that $\Gen(T_\mathcal{G})=\Div{\mathcal{G}}$.
 		
	   Since $C'_\mathcal{G}=C_\mathcal{G}/C_0$, it is enough to show $\Gen(C_\mathcal{G})=\Div{\mathcal{G}}$. In $C_\mathcal{G}$, the image of any basis element $\lambda$ is identified with the linear combination $\sum_{k < n_I} x_k^I (\lambda \sqcup (I,k))$ with coefficients from $I$, so $C_\mathcal{G} \in \Div{\mathcal{G}}$. Note that this implies that $C_\mathcal{G}$ is an $R/A$-module. For the other inclusion, let $M \in \Div{\mathcal{G}}$. It is clear that $AM=0$, and therefore there is a surjection $\pi: (C_0)^{(\varkappa)} \simeq (R/A)^{(\varkappa)} \rightarrow M$. Since $\Ext{1}{R/A}{S_I}{M}=0$ for every $I \in \mathcal{I}$ by Step II, we have by Step III and by the Eklof Lemma that $\Ext{1}{R/A}{C'_\mathcal{G}}{M}=0$, and thus the $R/A$-homomorphism $\pi$ can be extended to a (surjective) map $C_\mathcal{G}^{(\varkappa)} \rightarrow M$, proving the claim.
	   
	   {\bf Step V:} The map $\varphi_\mathcal{G}$ induces 
	    a commutative diagram with exact rows
		 $$
	   	\begin{CD}
				& & K @>{\varphi_\mathcal{G}}>> F @>>> C_\mathcal{G} @>>> 0 \\
				& & {}@VVV {}@VVV {}@VVV \\
				0 @>>> K/AK @>{\overline{\varphi_\mathcal{G}}} >> F/AF  @>>> C_\mathcal{G} @>>>0 \\
		\end{CD}
	   $$
	   and the analagous result holds for $\varphi'_\mathcal{G}$.
	   	
Indeed, the $R/A$-module  $C_\mathcal{G}$ is the cokernel of  ${\overline{\varphi_\mathcal{G}}}$. Further, since $K/AK$ is a free $R/A$-module with basis 
	    $\Lambda \times \mathcal{I}$,   injectivity of ${\overline{\varphi_\mathcal{G}}}$ amounts  to showing that the elements $\overline{\varphi_\mathcal{G}((\lambda, I))}$ with $(\lambda, I)\in\Lambda \times \mathcal{I}$ form an $R/A$-linearly independent subset in $F/AF$. To this end, we prove in next paragraph that  for each $n \in \omega$
	the elements $\overline{\varphi_\mathcal{G}((\lambda, I))}$  where $\lambda$ has length $n$ form a linearly independent subset in
	the  free $R/A$-module   
	$F_{n+1}/F_n\otimes_R R/A$
	with basis   $\Lambda_{n+1} \setminus \Lambda_{n}$. 
	Then indeed, as $\varphi_\mathcal{G}(\Lambda_{n-1} \times \mathcal{I}) \subseteq F_n$ for each $n>0$, a simple induction argument shows the linear independence of the $\varphi_\mathcal{G}$-image of $\Lambda_n \times \mathcal{I}$ for each $n \in \omega$, and thus of the $\varphi_\mathcal{G}$-image of the whole basis $\Lambda \times \mathcal{I}$. 
	   
	   For any sequence $\lambda \in \Lambda$ of length $n$ and any $I \in \mathcal{I}$, the image of $(\lambda,I)$ in $F_{n+1}/F_n \otimes_R R/A$ is equal to $\sum_{k<n_I}(x_k^I + A)(\lambda \sqcup (I,k))$. As these elements are linear combinations of pairwise disjoint subsets of $\Lambda$, it is clear that their spans are independent in the free $R/A$-module $F \otimes_R R/A$  with basis $\Lambda$. To prove $R/A$-linear independency, it remains to show that these elements have zero annihilator over $R/A$. But that follows from Step I, as $\operatorname{Ann}_{R/A}(\sum_{k<n_I}(x_k^I + A)(\lambda \sqcup (I,k)))=\operatorname{Ann}_{R/A}((I+A)/A)=0$.

	   So ${\overline{\varphi_\mathcal{G}}}$ is injective, and the proof of injectivity of ${\overline{\varphi'_\mathcal{G}}}$ is completely analogous. 

	   {\bf Step VI:} $\mathcal{D}_{\sigma_\mathcal{G}}=\Div{\mathcal{G}}$.

	   Let  $M \in \mathcal{D}_{\sigma_\mathcal{G}}$. We first show that $AM=0$. For any $m \in M$, define map $f: K \rightarrow M$ by setting $f( (\lambda, I) )=m$ for each $(\lambda, I) \in \Lambda \times \mathcal{I}$. As $\mathcal{D}_{\sigma_\mathcal{G}} \subseteq \mathcal{D}_{\varphi'_\mathcal{G}}$, there is a map $g: F' \rightarrow M$ such that $f=g\varphi'_\mathcal{G}$. But $\varphi'_\mathcal{G}( (w, I ) ) = \sum_{k<n_I}x_k^I (I,k)$ is annihilated by $\Ann{I}$. It follows that $\Ann{I}M=0$ for all $I \in \mathcal{I}$, and thus $AM=0$.
 Now, since  $M \in \mathcal{D}_{\sigma_\mathcal{G}}$ also implies $\Ext{1}{R}{C'_\mathcal{G}}{M}=0$ by Lemma~\ref{L00}, we can conclude  as in Step IV that 
	    $M\in \Gen C_\mathcal{G}=\Div{\mathcal{G}}$.

	  Conversely,  let  $M \in \Div{\mathcal{G}}$, and let $f: P_{-1} \rightarrow M$ be a map. By Step V we have a commutative diagram with exact rows
	   $$
	   	\begin{CD}
				& & P_{-1} @>{\sigma_\mathcal{G}}>> P_0 @>>> T_\mathcal{G} @>>> 0 \\
				& & \pi@VVV \psi@VVV @| \\
				0 @>>> P_{-1}/AP_{-1} @>{\overline{\sigma_\mathcal{G}}}>> P_0/AP_0 @>>> T_\mathcal{G} @>>>0, \\
		\end{CD}
	   $$
	   where the vertical maps $\pi$ and $\psi$ are the canonical projections.
	    As $M$ is an $R/A$-module, the map $f$ can be factorized through $\pi$, say $f=f'\pi$. Now $T_\mathcal{G}$ is $\{R/A\} \cup \mathcal{S}$-filtered by Step III, so the Eklof Lemma and Step II imply $\Ext{1}{R/A}{T_\mathcal{G}}{M}=0$. Therefore, there is a map $h: P_0/AP_0 \rightarrow M$ such that $f'=h \,\overline{\sigma_\mathcal{G}}$. Then $f=h\psi\sigma_\mathcal{G}$, proving that $M \in \mathcal{D}_{\sigma_\mathcal{G}}$.
\end{proof}
\begin{theorem}
	\label{T01}
	Let $R$ be a commutative ring. There is a 1-1 correspondence between
	\begin{enumerate}
			\item[(i)] silting classes $\mathcal{D}$ in $\ModR$,
			\item[(ii)] Gabriel filters of finite type $\mathcal{G}$ over $R$.
	\end{enumerate}
	The correspondence is given as follows:
	$$\Theta: \mathcal{G} \mapsto \Div{\mathcal{G}},$$
	$$\Xi: \mathcal{D} \mapsto \{I \subseteq R \mid M=IM \text{ for all } M \in \mathcal{D}\}.$$
\end{theorem}
\begin{proof}
		By 
		Proposition~\ref{P01} and Proposition~\ref{P02}, both maps of the correspondence are well defined. By Proposition~\ref{P01}, it is clear that $\Theta(\Xi(\mathcal{D}))=\mathcal{D}$ for any silting class $\mathcal{D}$. That $\Xi(\Theta(\mathcal{G}))=\mathcal{G}$ for any Gabriel topology of finite type follows from \cite[Chapter VI, Theorem 5.1]{St}, and by character duality.
\end{proof}
In \cite{SP}, it was asked whether any tilting torsion pair $(\mathcal{T},\mathcal{F})$ of finite type is classical (that is, $\mathcal{T}$ is generated by a finitely presented tilting module). The answer turned out to be negative for general rings, but positive for commutative rings (\cite{BHPST}). We remind that for commutative rings, this means that $\mathcal{F}$ is closed under direct limits if and only if $\mathcal{F}=\{0\}$. We conclude this section with a generalization of this phenomenon for silting classes. 
\begin{proposition}
		\label{P11}
		Let $R$ be a commutative ring, $T$ a silting $R$-module, and $(\mathcal{D}, \mathcal{F}) = (\Gen(T), \Ker \Hom{R}{T}{-})$ the associated torsion pair. The following are equivalent:
	\begin{enumerate}
		\item[(i)] $T$ is projective,
		\item[(ii)] there is a finitely presented silting $R$-module generating $\mathcal{D}$,
		\item[(iii)] $\mathcal{F}$ is closed under direct limits,
		\item[(iv)] $\mathcal{D} = \Gen(Re)$ for a (central) idempotent $e \in R$.
	\end{enumerate}
\end{proposition}
\begin{proof}
		Denote $A=\Ann{T}=\Ann{\mathcal{D}}$. By \cite[Proposition 3.13 and Lemma 3.4]{AMV}, $T$ is a tilting $R/A$-module.
Moreover, it is easy to check that $R/A\in\mathcal{D}$ if and only if $\mathcal{D}=\rmod{R/A}$, or equivalently, $\mathcal{D}=\Gen{R/A}$. In this case,  $A$ is  an idempotent ideal with $\mathcal{D}=\Ker \Hom{R}{A}{-}\subset (R/A)^\perp$, cf.~ \cite[Proposition 2.5]{AMV2}.

		(i) $\rightarrow$ (iv): As $T$ is also projective as an $R/A$-module, $\mathcal{D}=\Ker \Ext{1}{R/A}{T}{-}=\rmod{R/A}$. Then $R/A \in \Add(T)$ 
		is a projective $R$-module. Hence, $R/A=Re$ for an idempotent $e \in R$, and (iv) follows.
		
		(ii) $\rightarrow$ (iv): Let $T'$ be a finitely presented silting module such that $\Gen(T')=\mathcal{D}$. Then $T'$ is a finitely presented tilting $R/A$-module, which is projective by \cite[Proposition 13.2]{GT}). Hence, $\mathcal{D}=\rmod{R/A}$, and $A$ is an idempotent ideal. Also,  $R/A \in \Add(T')$ is finitely presented, and thus $A$ is finitely generated. 		
It follows from \cite[Proposition 1.10(i)]{FS} that $A=Rf$ for some idempotent $f \in R$, and thus $R/A=R(1-f)$, proving (iv).

		(iii) $\rightarrow$ (iv): Consider the (tilting) torsion pair $(\mathcal{D}, \mathcal{F}')$ in $\rmod{R/A}$, where $\mathcal{F}'=\Ker \Hom{R/A}{T}{-}$. Then $\mathcal{F}'$ is closed under direct limits, and thus $\mathcal{D}=\rmod{R/A}$ by \cite{BHPST} or \cite[Theorem 4.6]{H}. In particular, $A$ is an idempotent ideal. 
		
		We claim that $A$ is finitely generated. Let us write $A$ as a direct union of its finitely generated subideals, $A=\varinjlim_{j \in J} I_j$. Denote by $K_j$ the ideal such that $R/K_j$ is the torsion-free quotient of $R/I_j$ with respect to the torsion pair $(\mathcal{D},\mathcal{F})$. Then $K_i \subseteq K_j$ whenever $i \leq j \in J$. Since $\mathcal{F}$ is closed under direct limits, we have that $\varinjlim_{j \in J} R/K_j = R/\bigcup_{j \in J} K_j$ is in $\mathcal{F}$, and thus zero, because $R/A \in \mathcal{D}$, and $A \subseteq \bigcup_{j \in J}K_j$. It follows that there is $j \in J$ such that $R=K_j$, and therefore $R/I_j \in \mathcal{D}$. But $\mathcal{D} = \rmod{R/A}$, and $I_j \subseteq A$, which forces $I_j=A$.

		We now conclude this implication as in (ii) $\rightarrow$ (iv).
		
		(iv) $\rightarrow$ (i), (ii), (iii): As $\mathcal{F} = \Gen(R(1-e))$, condition (iii) is clear. Consider the map $\sigma: R \rightarrow Re \oplus Re$ given by  the canonical projection of $R$ onto the first direct summand $Re$. Then 
		$\Ker(\sigma)=R(1-e)$, and clearly $\mathcal{D}_\sigma = \Ker \Hom{R}{R(1-e)}{-} = \Gen(Re)$. Hence, $\Coker(\sigma) = Re$ is a silting module generating $\Gen(Re)$. This proves (ii). Finally, $T \in \Add{Re}$, and thus $T$ is projective.
\end{proof}
The following example shows that, in contrast with tilting modules over commutative rings, we cannot replace ``finitely presented'' by ``finitely generated''  in Proposition~\ref{P11}(ii).
\begin{example}
		Let $k$ be a field, $\varkappa$ an infinite cardinal, and $R=k^\varkappa$. Consider the Gabriel filter $\mathcal{G}$ over $R$ with basis consisting of all principal ideals generated by elements of $k^\varkappa$, such that their support is cofinite in $\varkappa$. Let $\mathcal{D}=\Div{\mathcal{G}}$ be the associated silting class. Then $A=\Ann{\mathcal{D}}=\sum_{I \in \mathcal{G}}\Ann{I}$ is equal to $k^{(\varkappa)} \subseteq R$. Because $A+I=R$ for any $I \in \mathcal{G}$, we have that $R/A \in \mathcal{D}$, and therefore $\mathcal{D} = \Gen{R/A}=\Ker \Hom{R}{A}{-}\subset (R/A)^\perp$ (cf.~the proof of Proposition \ref{P11}). We claim that $R/A$ is a silting module.

		For each $a \in \varkappa$, consider the idempotent $e_a\in R$ with $a$-th component equal to $1$, and all other components equal to zero. Taking the direct sum of the split exact sequences $0\to e_aR\to R\to (1-e_a)R
		\to 0$, we obtain a split exact sequence $0\to A\stackrel{\iota}{\rightarrow} R\stackrel{\pi}{\rightarrow}\bigoplus_{a \in \varkappa}(1-e_a)R\to 0$, where $\mathcal{D}_\pi = \Ker \Hom{R}{A}{-}=\mathcal{D}\subset (R/A)^\perp$.
The map $\sigma = \iota \oplus \pi\in\Mor{\ProjR}$ then satisfies $\Coker(\sigma)=\Coker(\iota)=R/A$, and as $\iota$ is monic, $\mathcal{D}_\sigma=(R/A)^\perp \cap \mathcal{D}_{\pi}=\mathcal{D}_\pi=\Gen R/A$.	
		We proved that $R/A$ is silting.
		
		Finally, note that $R/A$ is not finitely presented, and thus not projective.
\end{example}
\section{Cosilting modules over commutative rings}

If $R$ is a commutative noetherian ring, then all Gabriel filters and all hereditary torsion pairs are of finite type, and they correspond bijectively to subsets of $\Spec{R}$ closed under specialization. Recall that a subset $P\subset\Spec{R}$ is {\it closed under specialization} if $\mathfrak p\in P$ implies that all prime ideals $\mathfrak q\supset\mathfrak p$ belong to $P$.  
Such  $P$   gives rise to  a  hereditary torsion pair   $(\mathcal T(P),\mathcal F(P))$  where
$\mathcal F(P)=\{M\in \ModR\mid \Hom{R}{R/p}{M}=0\text{ for\ all }p\in P\}$, and the assignment $P\mapsto (\mathcal T(P),\mathcal F(P))$ defines the stated bijection.
For details we refer to \cite[Chapter VI, \S 6.6]{St}.

\begin{theorem} \label{classcosilting}
If $R$ is a commutative noetherian ring, there are bijections between
\begin{ii}
\item silting classes $\mathcal{D}$ in $\ModR$,
\item subsets $P \subseteq \Spec R$ closed under specialization,
\item Gabriel filters $\mathcal{G}$ over $R$,
\item cosilting classes $\C$ in $\rmod R$.
\end{ii}
In particular, every cosilting class is of cofinite type.
\end{theorem}
\begin{proof} 
Apply Corollary~\ref{coftnoe} and Theorem~\ref{T01}.
\end{proof}

Next, we provide a  construction for a cosilting module
cogenerating the $\mathcal{G}$-torsionfree modules for a given Gabriel filter $\mathcal{G}$. It is inspired by the construction of cotilting modules over commutative noetherian rings  in \cite{HST}.
\begin{construction}\label{coscon}
		Let $R$ be commutative, and $\mathcal{G}$ be a Gabriel filter of finite type. Let $(\mathcal{T}_\mathcal{G}, \mathcal{F}_\mathcal{G})$ be the associated hereditary torsion pair from \ref{reminder}, that is, $\mathcal{F}_\mathcal{G}=\bigcap_{I \in \mathcal{G}}\Ker \Hom{R}{R/I}{-}$, and $\mathcal{T}_\mathcal{G}$ consists of the modules $M$ for which every element  $ m \in M$ is annihilated by some $I \in \mathcal{G}$. Let us construct a cosilting module $C_\mathcal{G}$ such that $\Cogen(C_\mathcal{G})=\mathcal{F}_\mathcal{G}$.

First, since $\mathcal{F}_\mathcal{G}$ is a hereditary torsion-free class, there is an injective module $E$ with $\Cogen(E)=\mathcal{F}_\mathcal{G}$. Indeed, we can put $E=\prod\{E(R/J) \mid R/J \in \mathcal{F}_\mathcal{G}\}$. Then $E$ is injective, $E \in \mathcal{F}_\mathcal{G}$, and any module from $\mathcal{F}_\mathcal{G}$ is easily seen to be cogenerated by $E$. 

Next, we let $E_1=\prod\{E(R/I) \mid I \in \mathcal{G}\}$. Since $\mathcal{G}$ is of finite type, $\mathcal{F}_\mathcal{G}$ is definable, and thus a precovering class. Let $f: F \rightarrow E_1$ be a $\mathcal{F}_\mathcal{G}$-precover of $E_1$. Since $E_1$ is injective, we can extend $f$ to a map $\bar{f}: E(F) \rightarrow E_1$. As $E(F) \in \mathcal{F}_\mathcal{G}$, the map $\bar{f}$ is also an $\mathcal{F}_\mathcal{G}$-precover of $E_1$. Then $E_0=  E \oplus E(F) $ is an injective module in $\mathcal{F}_\mathcal{G}$. Denote by $k: K \rightarrow E(F)$  the kernel of $\bar{f}$, and consider the following exact sequence:
$$0 \rightarrow E \oplus K \xrightarrow{\begin{pmatrix} 1_E & 0 \\ 0 & k \end{pmatrix}} E_0\xrightarrow{\begin{pmatrix}0 & \bar{f}\end{pmatrix}} E_1.$$
We claim that $C_\mathcal{G}=E \oplus K$ is a cosilting module with respect to the map $\lambda=\begin{pmatrix}0 & \bar{f}\end{pmatrix}$, and $\Cogen(C_\mathcal{G})=\mathcal{F}_\mathcal{G}$. 

Since $\Cogen(E)=\mathcal{F}_\mathcal{G}$, and $K$ is isomorphic to a submodule of $E(F) \in \mathcal{F}_\mathcal{G}$, we have $\Cogen(C_\mathcal{G})=\mathcal{F}_\mathcal{G}$. 
Further, if $M \in \mathcal{F}_\mathcal{G}$, then any map $g: M \rightarrow E_1$ factors through the  $\mathcal{F}_\mathcal{G}$-precover $\bar{f}$ of $E_1$, so there is $h: M \rightarrow E(F)$ such that  $g=\bar{f}h=\lambda\begin{pmatrix}0 \\ h\end{pmatrix}$, and thus $M \in \mathcal{C}_\lambda$. This proves that $\mathcal{F}_\mathcal{G}\subseteq\mathcal{C}_\lambda$.

		Let now $M$ be an $R$-module such that the $\mathcal{T}_\mathcal{G}$-torsion part $M'$ of $M$ is non-zero. Choose any non-zero cyclic submodule $R/I$ of $M'$. As necessarily $I \in \mathcal{G}$, there is a non-zero map $g: R/I \rightarrow E_1$, which extends to $\bar{g}: M \rightarrow E_1$. Suppose that there is $h: M \rightarrow E_0 $ such that $\bar{g}=\lambda h$. Then $h_{\restriction M'}$ is a non-zero map $M' \rightarrow E_0$ with $E_0 \in \mathcal{F}_\mathcal{G}$, a contradiction. Therefore, $M \not\in \mathcal{C}_\lambda$. We have $\mathcal{F}_\mathcal{G}=\mathcal{C}_\lambda$ as desired.
\end{construction}

\begin{corollary}
Let $R$ be a commutative ring. With the notation of Constructions \ref{CC00} and \ref{coscon}, $\{T_\mathcal{G} \,\mid\, \mathcal{G} \text{ a  Gabriel filter of finite type }\}$ is a set of representatives, up to equivalence, of all silting $R$-modules, and $\{C_\mathcal{G} \,\mid\, \mathcal{G} \text{ a  Gabriel filter of finite type }\}$ is  a set of representatives, up to equivalence, of all cosilting $R$-modules of cofinite type. 
\end{corollary}

We close this note with an example  of a cosilting module  which is not of cofinite type. The same  module is also an example for a finendo quasitilting module which is not silting. 
Recall that all silting modules are finendo quasitilting  \cite[Proposition 3.10]{AMV}. 

\begin{example}
		\label{E00}
		Let $R$ be a commutative local ring with a non-zero idempotent maximal ideal $\mathfrak{m}$ (e.g.~any valuation domain with non-zero idempotent radical, such as the ring of Puiseux series over a field). 
We consider the module $R/\mathfrak{m}$.

Since $\mathfrak{m}$ is idempotent, the class $\mathcal{C}=\Gen(R/\mathfrak{m})=\Add(R/\mathfrak{m})$ is a torsion class
contained in $(R/\mathfrak{m})^\perp$. The natural projection $R \rightarrow R/\mathfrak{m}$ is easily seen to be a $\mathcal{C}$-preenvelope. The cokernel of this map is zero, and  \cite[Proposition 3.2]{AMV} shows that $R/\mathfrak{m}$ is a finendo quasitilting module. On the other hand, $\mathcal{C}$ is not silting by Theorem~\ref{T01}. Indeed, the only ideal $R/\mathfrak{m}$ is divisible by is $R$. But $\mathcal{C} \neq \ModR$, because $\mathfrak{m} \not\in \mathcal{C}$, as $\mathfrak{m}^2=\mathfrak{m} \neq 0$. 

The same class  $\mathcal{C}$ is a cosilting class not of cofinite type. Indeed, $\mathcal{C}$ is  closed for direct products, and thus it coincides with $\Cogen(R/\mathfrak{m})$.
We prove that
$R/\mathfrak{m}$ is a cosilting module.
Let $0 \rightarrow R/\mathfrak{m} \rightarrow E_0 \xrightarrow{\varphi} E_1$ be the begining of the minimal injective coresolution of $R/\mathfrak{m}$. Define an injective module $E=\prod \{E(R/J) \mid J \subseteq R \text{ such that } \Soc{R/J}=0\}$. Let $\sigma: E_0 \rightarrow E_1 \oplus E$ be the direct sum of $\varphi$ and the zero map $0 \rightarrow E$. We prove that $\mathcal{C}_\sigma=\mathcal{C}$.

		Note that the image of any map $f: R/\mathfrak{m} \rightarrow E_1 \oplus E$ is contained in $E_1$  by the definition of $E$. By the essentiality of the image of $\varphi$ in $E_1$, $f$ is actually a map $R/\mathfrak{m} \rightarrow \im{\varphi}$. Since $\Ext{1}{R}{R/\mathfrak{m}}{R/\mathfrak{m}}=0$ by the idempotency of $\mathfrak{m}$, we have that $R/\mathfrak{m} \in \mathcal{C}_\sigma$, and thus $\mathcal{C} \subseteq \mathcal{C}_\sigma$.

		Let now $M \in \ModR$ be such that $\mathfrak{m}M \neq 0$. Then $M$ contains a cyclic submodule $R/I$ with $\mathfrak{m} \not\subseteq I$. Using injectivity, it is enough to show that $R/I \not\in \mathcal{C}_\sigma$. If $\Soc{R/I}=0$, then $R/I$ injects into $E$, and this injection clearly cannot be factorized through $\sigma$. If $\Soc{R/I}\neq 0$, let $J$ be an ideal such that $(R/I)/\Soc{R/I} \simeq R/J$. Then $J \neq R$, because in such case $\Soc{R/I}=R/I$, implying that $\Ann{R/I}=\mathfrak{m}$, and thus $R/I=R/\mathfrak{m}$, which is not the case. If $\Soc{R/J} \neq 0$, the full preimage of this socle in $R/I$ would be a non-trivial extension of two semisimple modules, which does not exist by idempotency of $\mathfrak{m}$. Hence, $\Soc{R/J}=0$, and the composition of the projection $R/I \rightarrow R/J$ with inclusion $R/J \rightarrow E$ is a non-zero map $R/I \rightarrow E$. Again, this map cannot be factorized through $\sigma$. Hence, $R/I \not\in \mathcal{C}_\sigma$, and $\mathcal{C}_\sigma=\Cogen(R/\mathfrak{m})$.

		Finally, the class $\mathcal{C}$ is not of cofinite type. Indeed, the only injective the class $\mathcal{C}$ contains is zero, and thus it is not of cofinite type by Lemma~\ref{L35}.		
\end{example}


\begin{thebibliography}{AHT}
\bibitem{AIR} {\sc T.~Adachi, O.~Iyama, I.~Reiten}, {$\tau$-tilting theory}, Compos. Math. {\bf 150} (2014), 415--452.
\bibitem{AF}{\sc F.W. Anderson, K.R. Fuller}, \textit{Rings and Categories of Modules}, Second Edition, Springer-Verlag, 1992. 
\bibitem{AMV}{\sc L.~Angeleri H\"ugel, F. Marks, J. Vit\'oria}, Silting modules, International Mathematics Research Notes 2015, doi\,10.1093/imrn/rnv191.

\bibitem{AMV2}{\sc L.~Angeleri H\"ugel, F. Marks, J. Vit\'oria}, Silting modules and ring epimorphisms, 
 preprint, arXiv:1504.07169.
\bibitem{APST}{\sc  L.~Angeleri H\"ugel, D.~Posp{\'\i}{\v s}il, J.~\v{S}\v{t}ov\'\i\v{c}ek, J.~Trlifaj}, Tilting, cotilting, and spectra of commutative noetherian rings. {Transactions  Amer. Math. Soc.} {\bf 366} (2014), 3487--3517.





\bibitem{B}{\sc S. Bazzoni}, {Cotilting and tilting modules over Pr\"ufer domains}, Forum Math. \textbf{19} (2007), 10005--1027.  

\bibitem{BHPST}{\sc S. Bazzoni, I. Herzog, P. P\v{r}\'{i}hoda, J.\v{S}aroch, J. Trlifaj}, in preparation.


\bibitem{BP}{\sc S. Breaz, F. Pop}, Cosilting modules, preprint, arXiv:1510.05098.





%






\bibitem{FS}{\sc L. Fuchs, L. Salce}, \textit{Modules over Non-Noetherian Domains}, AMS, 2001.

\bibitem{GT}{\sc R. G\"obel, J. Trlifaj}, \textit{Approximations and Endomorphism Algebras of Modules}, GEM \textbf{41}, W. de Gruyter, Berlin 2006.


\bibitem{HST}{\sc D. Herbera, J. J.~\v{S}\v{t}ov\'\i\v{c}ek, J.~Trlifaj}, Cotilting modules over  commutative noetherian rings. Journal of pure and applied algebra {\bf 218} (2014), 1696-1711.
\bibitem{H}{\sc M. Hrbek}, One-tilting classes and modules over commutative rings, preprint, arXiv:1507.02811. 
\bibitem{MS1} {\sc F. Marks, J. \v{S}\v{t}ov\'\i\v{c}ek}, {Torsion classes, wide subcategories and localisations}, preprint, arXiv:1503.04639.
\bibitem{MS}{\sc F. Marks, J.~\v{S}\v{t}ov\'\i\v{c}ek}, in preparation.

\bibitem{P}{\sc M. Prest}, {\it Purity, spectra and localisation}, Cambridge University Press 2009.




\bibitem{SP}{\sc C. E. Parra, M. Saor{\'\i}n}, Direct limits in the heart of a t-structure: the case of a torsion pair. Journal of pure and applied algebra {\bf 219} (2015), 4117-4143.
\bibitem{St} {\sc B. Stenstr\o m}. \emph{Rings of quotients}. Springer-Verlag (1975).




\bibitem{WZ} {\sc J.~Wei, P.~Zhang} Cosilting complexes and AIR-cotilting modules, preprint, arXiv 1601.01385


\end{thebibliography}
\end{document}